\documentclass[]{interact}

\usepackage[numbers,sort&compress]{natbib}
\bibpunct[, ]{[}{]}{,}{n}{,}{,}
\renewcommand\bibfont{\fontsize{10}{12}\selectfont}

\usepackage{caption}
\usepackage{algorithmic}
\usepackage[ruled,linesnumbered,vlined]{algorithm2e}
\usepackage{hyperref}
\usepackage{multirow}
\usepackage{subcaption}
\newdimen\figrasterwd
\figrasterwd\textwidth
\usepackage[shortlabels]{enumitem}
\usepackage{xcolor}
\usepackage{graphicx}

\theoremstyle{plain}
\newtheorem{theorem}{Theorem}[section]

\newtheorem{corollary}[theorem]{Corollary}
\newtheorem{proposition}[theorem]{Proposition}

\theoremstyle{definition}

\newtheorem{assumption}{Assumption}

\theoremstyle{remark}
\newtheorem{remark}{Remark}

\begin{document}

\articletype{ARTICLE}

\title{Penalty decomposition derivative free method for  the minimization of partially separable functions over a convex feasible set}

\author{
\name{Francesco Cecere\textsuperscript{a}, Matteo Lapucci\textsuperscript{b}\thanks{CONTACT Davide Pucci. Email: davide.pucci@unifi.it}, Davide Pucci\textsuperscript{b},  Marco Sciandrone\textsuperscript{a}}
\affil{\textsuperscript{a}Dipartimento di Ingegneria Informatica, Automatica e Gestionale ``Antonio Ruberti'', Università di Roma ``La Sapienza'', Via Ariosto, 25, 00185, Roma, Italy; \textsuperscript{b}Dipartimento di Ingegneria dell'Informazione, Università di Firenze, Via di S.\ Marta, 3, 50135, Firenze, Italy}
}

\maketitle
\begin{abstract}
In this paper, we consider the problem of minimizing a smooth function, given as finite sum of black-box functions, over a convex set. In order to advantageously exploit the structure of the problem,
for instance when the terms of the objective functions are partially separable, noisy, costly or with first-order information partially accessible, we propose a framework where the penalty decomposition approach is combined with a derivative-free line-search-based method. Under standard assumptions, we state theoretical results showing that the proposed algorithm is well-defined and globally convergent to stationary points. The results of preliminary numerical experiments, performed on test problems with number of variables up to thousands, show the validity of the proposed method w.r.t. state of the art methods.
\end{abstract}


\begin{amscode}
	90C56, 90C30, 90C26 
\end{amscode}
\section{Introduction}\label{sec1}
In this work, we are interested in finite-sum optimization problems of the form
\begin{equation}\label{prob_main}
 \begin{aligned}
     \min_x\;& f(x)=\displaystyle{\sum_{j=1}^m}f_j(x)\\
  \text{s.t. }& x\in X
 \end{aligned}
\end{equation}
where $X\subset \mathbb{R}^n$ is a nonempty, closed, convex set and $f_j:\mathbb{R}^n\to \mathbb{R}$, for $j=1,\ldots ,m$ are smooth functions for which, however, we do not have access to first-order information. We also assume that $f$ is bounded below by some value $f^*$ and that the euclidean projection operator onto $X$ is effectively available. The latter assumption would in principle hold for any convex set; yet, while projection can be exactly and cheaply obtained for important constraints sets such as bounds or spheres, in other cases it is impractical to handle.

The literature on derivative-free methods is wide and increasing. For insights on general derivative-free approaches we refer the reader to the books \cite{conn2009introduction,Audet2017}, the survey \cite{larson2019derivative} and to some seminal papers both for smooth optimization (e.g., \cite{de1984stopping,torczon1,torczon2,powell}) and nonsmooth optimization (e.g., \cite{audet2008,lucidi1,lucidi2019}). A very recent frontier (investigated in \cite{bollapragada2025derivative} and references therein) is that of derivative-free stochastic optimization, where the black-box provides evaluations of a stochastic objective function.

According to the classification of derivative-free methods stated in \cite{larson2019derivative}, we set our work within the class of {\it direct search methods}, i.e., methods that use only function values
without approximating the gradient. This specific class of method is also the subject of the thorough survey \cite{dzahini2025direct}.
More specifically, we deal with deterministic objective functions, and
the proposed derivative-free approach fits that of the so-called {\it directional direct-search methods}, where the current iterate $x_k$  is updated by a term of the form $\alpha_k d_k$, where $\alpha_k$ is the step size and $d_k$ is a vector taken from a finite set of search directions $D_k$ that surely includes a feasible descent direction, so that
(sufficiently) small step sizes along such direction lead to an improved feasible solution. Most directional direct-search methods share the choice of the set of search directions $D_k$ and differ in the strategies to sample the objective
function along them. We adopt the line-search strategy, i.e., a new point is
produced by performing a derivative free line-search that guarantees a sufficient decrease of the objective
function.

For strongly structured problems, like \eqref{prob_main}, tailored derivative-free methods, have been presented; for instance, in \cite{cristofari2021derivative} authors exploit the peculiarities of the feasible set (being the convex hull of a given set of atoms) of the problem; partial knowledge of first order information is taken into account in \cite{liuzzi2012decomposition}; the partially separable structure of the objective function is addressed in \cite{PorcelliToint2022Exploiting}.

Drawing inspiration from the latter work, we have considered \eqref{prob_main}, noting that its structure may deserve attention in designing derivative-free algorithms, more in general, in the following cases:
\begin{enumerate}[(i)]
 \item the objective function is {\it coordinate partially separable} (CPS), i.e., each function $f_j$ only depends on a (relatively small) subset of variables $S_j\subseteq\{1,\ldots,n\}$;
 \item  the evaluation of some (or all) terms of the objective function may be costly;
 \item  the functions $f_j$ may be affected by noise differently as result of different simulations or measurements;
 \item  the objective function is {\it partially derivative-free}, i.e., for some functions $f_j$ the gradient is available.
 \end{enumerate}
 In order to take into account the structure of problem \eqref{prob_main} related to the possible cases (i)-(iv),
 we adopt a {\it penalty decomposition} approach
 that allows to manage cases (i)-(iv) in a unified way.

 Briefly, the penalty decomposition approach consists in the introduction of an auxiliary vector of variables - copies in principle of the original ones - and the addition of a (sequential) penalty term to the objective function, penalizing the distance between each pair of corresponding variables.
 The introduction of the auxiliary variables allows the decomposition of the problem into smaller independent subproblems; this
may be advantageously exploited to take into account the peculiarities of problem \eqref{prob_main} in the cases (i)-(iv) discussed above.

The penalty decomposition approach is related to that of Alternating Direction Method of Multipliers (ADMM)  (see, e.g., \cite{neal2011distributed}), which could be used to solve the reformulated problem in place of the sequential penalty method. ADMM methods come with a strong convergence theory and perform well in practice even in some nonconvex cases, but require to compute the global minimizers of augmented Lagrangian functions. Therefore, the development and application of ADMM-based methods may be prohibitive in a general nonconvex and derivative-free setting as that of the present work where, to take into account these issues, we adopt the sequential penalty approach.

In recent years, various structured problems have been handled by a penalty decomposition strategy; for instance, penalty decomposition allows for an efficient management of sparsity constraints \cite{lu2013sparse}  and more generally of geometric constraints \cite{kanzow2023inexact}, or to manage the intersection of convex sets \cite{galvan2020alternating}. In all of these cases, the availability of first order information was assumed.
To our knowledge, derivative-free penalty decomposition methods have not yet been studied,  except for the case of \cite{lapucci2021convergent}, where a method for sparse optimization has been designed to minimize a smooth black-box function to a cardinality constraint.

In this work we define a general algorithm model where the penalty decomposition strategy is coupled with a derivative-free line-search-based method for the management of subproblems.
We state theoretical results - not immediately descending from known results in the literature - for proving global convergence properties of the proposed algorithm and we present its characterization in the specific case of partially separable functions. Finally, we report the results of preliminary computational experiments.

The paper is organized as follows: Section \ref{sec2} concerns the penalty decomposition reformulation of problem \eqref{prob_main}, with the convergence analysis of the sequential scheme discussed in Section \ref{sec21}; the inner optimization scheme is presented in \ref{sec3} and formally analyzed in \ref{sec31}, whereas the discussion on the stopping condition is provided in \ref{sec32}; in Section \ref{sec4} we address the particular case of CPS problems; numerical experiments are described and discussed in Section \ref{sec5}; we finally give concluding remarks in Section \ref{sec6}.

\subsubsection*{Notation}
We denote by $\Pi_X$ the Euclidean projection operator onto the convex set $X$. A point $x\in X$ is stationary for problem \eqref{prob_main} if $x=\Pi_X[x-\nabla f(x)]$. We denote by $e_i$ the $i$-th element of the canonical basis in $\mathbb{R}^n$. If $S\subseteq\{1,\ldots,n\}$ is an index set, we denote by $x_S$ the subvector of variables of $x$ indexed by $S$.

\section{A Penalty decomposition approach}\label{sec2}
To describe the methodology proposed in this paper and also for subsequently carrying out its convergence analysis, the following problem, that allows for a much simpler notation, can be considered without loss of generality:  
\begin{equation}\label{prob_main_3}
 \begin{aligned}
        \min_{x}\;& f(x)=f_1(x)+f_2(x)\\
  \text{s.t. }&x\in X.
\end{aligned}
\end{equation}

\noindent Following a well-established approach \cite{lu2013sparse,galvan2020alternating,lapucci2021convergent,kanzow2023inexact}, the problem can be reformulated as 
\begin{equation}
    \label{prob:vincolato}
    \begin{aligned}
        \min_{x,y,z}\;&f_1(y)+f_2(z)\\\text{s.t. }&x\in X,\\&x-y=0,\\&x-z=0,
    \end{aligned}
\end{equation}
which is clearly equivalent to \eqref{prob_main_3}. Problem \eqref{prob:vincolato} then allows to introduce the \textit{penalty function}
$$P_\tau(x,y,z)=f_1(y)+f_2(z)+ \frac{\tau}{2} \left(\Vert x-y \Vert^2+\Vert x-z \Vert^2 \right),$$
whose (partial) gradients are given by
\begin{equation}
\label{b11}
    \begin{aligned}
        \nabla_x P_\tau (x, y,z) &= \tau \left(x-y\right)+\tau \left(x-z\right),\\
\nabla_y P_\tau (x, y,z,w) &= \nabla_x f _1(y ) - \tau \left(x-y\right),\\
\nabla_z P_\tau (x, y,z,w) &= \nabla_x f _2(z ) - \tau \left(x-z\right).
    \end{aligned}
\end{equation}
A sequential penalty method \cite[Ch.\ 4.2]{bertsekas1999nonlinear} can then be employed to address problem \eqref{prob:vincolato} and thus \eqref{prob_main_3}, which iteratively solves (inexactly) subproblems of the form
\begin{equation*}
\label{prob:subprob}
\begin{aligned}
    \min_{x, y,z \in \mathbb{R}^n}\;&  P_{\tau_k}(x,y,z)\\
    &\text{s.t. }x\in X,    
\end{aligned}
\end{equation*}
for a sequence of values $\{\tau_k\}$ such that $\tau_k\to\infty$. We shall highlight here that, at this stage, the inner solver that tackles the subproblems is an unspecified oracle, that could possibly employ first-order (ore even further) information. The black-box nature of the proposed overall scheme will be in the following sections.

The subproblem is considered (approximately) solved if a solution $(x^k,y^k,z^k)$ is found such that
\begin{align}
    \Vert \nabla_{y}P_{\tau_k}(x^k,y^k,z^k) \Vert &\le c_y \xi_k, \label{c11}\\
    \Vert \nabla_{z}P_{\tau_k}(x^k,y^k,z^k) \Vert &\le c_z\xi_k, \label{c12}\\
    \Vert x_k-\Pi_X[x_k-\nabla_{x}P_{\tau_k}(x^k,y^k,z^k)] \Vert &\le \xi_k, \label{c3}
\end{align}
for some positive tolerance $\xi_k>0$ and constants $c_y,c_z>0$.
The resulting algorithm is formally described by Algorithm \ref{alg:spm}. Note that, formally, the solution $(x^{k-1},y^{k-1},z^{k-1})$ is not used during iteration $k$; however, it is well-known that, with sequential penalty type methods, using the solution of the previous subproblem as a starting point for the next one is often good for efficiency and stability reasons \cite[Ch.\ 4.2]{bertsekas1999nonlinear}. We shall also observe that, in addition to the approximate optimality conditions, a further boundedness property is required for solution $(x^k,y^k,z^k)$.
\begin{remark}
    Note that, for any $\tau>0$ and any feasible $(x,y,z)$, we have
$$
P_\tau(x,y,z)=f(x).
$$
Therefore, if we start with a solution $(x^0,y^0,z^0)$ such that $x^0\in X$ and $y^0=z^0=x^0$, condition $P_{\tau_k}(x^k,y^k,z^k) \le P_{\tau_0}(x^0,y^0,z^0)$ can certainly be satisfied if at iteration $k$ we generate $(x^k,y^k,z^k)$ by means of a monotone descent method starting at
$$\begin{cases}
    (x^{k-1},y^{k-1},z^{k-1})&\text{if }P_{\tau_k}(x^{k-1},y^{k-1},z^{k-1})\le f(x^0),\\(x^0,y^0,z^0)&\text{otherwise}.
\end{cases}$$
\end{remark}
\noindent We now turn to the convergence analysis of the sequential penalty framework. 

\begin{algorithm}[h]
	\caption{\texttt{Sequential Penalty Approach}}
	\label{alg:spm}
	\textbf{Input:}  $x^0, y^0, z^0 \in \mathbb{R}^n,$ $\{\xi_k\}\subseteq \mathbb{R}^+, \{\tau_k\}\subseteq \mathbb{R}^+$, $c_y,c_z>0$\\
    \For{$k=1,2,\ldots$}{
    Find $(x^{k},y^{k},z^{k})\in X\times \mathbb{R}^n\times \mathbb{R}^n$ s.t.\ 
    \begin{gather*}
        P_{\tau_k}(x^k,y^k,z^k) \le P_{\tau_0}(x^0,y^0,z^0),\\
        \Vert \nabla_{y}P_{\tau_k}(x^k,y^k,z^k) \Vert \le c_y \xi_k,\\ \Vert \nabla_{z}P_{\tau_k}(x^k,y^k,z^k) \Vert \le c_z \xi_k,\\
        \Vert x_k-\Pi_X[x_k-\nabla_{x}P_{\tau_k}(x^k,y^k,z^k)] \Vert \le \xi_k.
    \end{gather*}
}
\end{algorithm}

\subsection{Convergence analysis}
\label{sec21}
We start the analysis stating the following assumption (formally given for the case of $f$ defined as the sum of $m$ terms).
\begin{assumption}
\label{ass:objective_function}
Functions $f_j:\mathbb{R}^n\to \mathbb{R}$, $j=1,\ldots, m$, are coercive.
\end{assumption}
We now begin stating the formal properties of the algorithm with a preliminary result.
\begin{proposition}
    \label{prop:pf_coercive}
    For any $\tau>0$, the penalty function $P_\tau(x,y,z)$ is coercive, i.e., $P_\tau(x^t,y^t,z^t)\to+\infty$ if $\|(x^t,y^t,z^t)\|\to\infty$.
\end{proposition}
\begin{proof}
    By Assumption \ref{ass:objective_function}, function $f(y,z)=f_1(y)+f_2(z)$ is also coercive. Assume $\{(x^t,y^t,z^t)\}$ is such that $\|(x^t,y^t,z^t)\|\to\infty$. If either $\|y^t\|\to\infty$ or $\|z^t\|\to\infty$, by the coercivity of $f(x,y)$, we get that 
    \begin{align*}
        \lim_{t\to\infty}P_\tau(x^t,y^t,z^t)&\ge \lim_{t\to\infty}f(y^t,z^t)=+\infty.
    \end{align*}
    On the other hand, if $\{(y^t,z^t)\}$ is bounded, it must be $\|x^t\|\to\infty$. Then, we have
    \begin{align*}
        \lim_{t\to\infty}P_\tau(x^t,y^t,z^t)&\ge \lim_{t\to\infty}f^*+\frac{\tau}{2}(\|y^t-x^t\|^2+\|z^t-x^t\|^2)=+\infty.
    \end{align*}
\end{proof}
\begin{corollary}
    \label{compactness}
Let $\bar{x}\in X$, $\bar{y},\bar{z}\in \mathbb{R}^{n}$, $\tau >0$.
The level set
$$
\mathcal{L}=\{(x,y,z) \in X \times \mathbb{R}^{n}\times \mathbb{R}^{n} \mid P_{\tau}(x,y,z) \le P_{\tau}(\bar{x},\bar{y},\bar{z})\}
$$
is compact.
\end{corollary}

We are now able to state a first property concerning the sequence $\{(x^k,y^k,z^k)\}$.
\begin{proposition}\label{boundedness}
Let $\{(x^k,y^k,z^k)\}$ be a sequence satisfying $P_{\tau_k}(x^k,y^k,z^k) \le P_{\tau_0}(x^0,y^0,z^0)$ for all $k$.
Then $\{(x^k,y^k,z^k)\}$ is a bounded sequence.
\end{proposition}
\begin{proof}
By Corollary \ref{compactness} we have that the set
$$
\mathcal{L}_0=\{(x,y,z) \in X \times \mathbb{R}^{n}
\times \mathbb{R}^{n}\mid  P_{\tau_0}(x,y,z) \le P_{\tau_0}(x^0,y^0,z^0)\}
$$
is compact.
The instructions of the algorithm imply that
$$
P_{\tau_0}(x^k,y^k,z^k)\le P_{\tau_k}(x^k,y^k,z^k)\le P_{\tau_0}(x^0,y^0,z^0),
$$
where the first inequality follows from the fact that $\tau_0\le \tau_k$.
Therefore, the  points of the sequence $\{(x^k,y^k,z^k)\}$ belong the compact set $\mathcal{L}_0$ and this concludes the proof.
\end{proof}

We finally state the main convergence result for Algorithm \ref{alg:spm}.
\begin{proposition}\label{prop_main}
    Let $\{(x^k,y^k,z^k)\}$ be the sequence generated by Algorithm \ref{alg:spm} with
    $\xi_k\to 0$ and $\tau_k\to\infty$.
    Then, $\{(x^k,y^k,z^k)\}$ admits limit points and every limit point $(x^*,y^*,z^*)$ is feasible, i.e., $x^*=y^*=z^*$, and $x^*$ is a stationary point of problem \eqref{prob_main_3}.
\end{proposition}
\begin{proof}
Consider the vectors $\mu_y^k,\mu_z^k\in \mathbb{R}^{n}$ defined as follows
\begin{equation}\label{vinc}
\mu_y^k = \tau_k (x_k-y_k),\qquad
\mu_z^k = \tau_k (x_k-z_k)
\end{equation}
Recalling \eqref{b11}, we therefore have
\begin{equation}\label{b1}
\begin{aligned}
    \nabla_x P_{\tau_k} (x^k, y^k,z^k) &= \mu_y^k+\mu_z^k,\\
\nabla_y P_{\tau_k} (x^k, y^k,z^k) &= \nabla_x f _1(y^k) -\mu_y^k, \\
\nabla_z P_{\tau_k} (x^k, y^k,z^k ) &= \nabla_x f _2(z^k ) -\mu_z^k.
\end{aligned}
\end{equation}
Since, by Proposition \ref{boundedness}, the sequence $\{(x^k,y^k,z^k)\}$ is bounded, it admits limit points.
Let $(x^*,y^*,z^*)$ be any limit point, i.e., there exists an infinite subset $K\subseteq \{0,1,\ldots \}$ such that
$$
\lim_{k\in K,\,k\to\infty}(x^k,y^k,z^k)=(x^*,y^*,z^*)
$$
Taking the limits in \eqref{b1} for $k\in K$ and $k\to \infty$, recalling \eqref{c11}-\eqref{c12} and that $\xi_k\to 0$, we obtain
\begin{equation}\label{mugy2}
\nabla_{x}f_1(y^*)=\mu^*_y=\lim_{k\to\infty,\,k\in K}\mu_y^k,\qquad
\nabla_{x}f_2(z^*)=\mu^*_z=\lim_{k\to\infty,\,k\in K}\mu_z^k,
\end{equation}
Furthermore,  we can divide both sides of the two equations in \eqref{vinc} by $\tau_k$ and take again the limits, to obtain that
$x^* = y^*= z^*$.
Therefore, we have
\begin{equation}\label{abc}
\mu^*_y+\mu^*_z=
\nabla_xf_1(y^*)+\nabla_xf_2(z^*)=
\nabla f(x^*).
\end{equation}
Using \eqref{b1}, condition \eqref{c3} can be written
as follows
\begin{equation}\label{c2bis}
 \Vert x_k-\Pi_X[x_k-\mu_{y}^k-\mu_{z}^k] \Vert \le \xi_k
\end{equation}
Then,
taking the limits in (\ref{c2bis}) for $k\in K$ and $k\to\infty$, recalling that $\xi_k\to 0$,  and using (\ref{abc}), we obtain
$$
x^*=\Pi_X[x^*-\mu_{y}^*-\mu_{z}^*]=\Pi_X[x^*-\nabla f(x^*)].
$$
\end{proof}

\begin{remark}
    The result in Proposition \ref{prop_main} cannot be directly inferred, to the best of our knowledge, by known results in the literature. Indeed, most known results concerning penalty and augmented Lagrangian methods rely on different stationarity conditions for subproblems with lower-level constraints to be handled explicitly out of the penalty function (see, e.g., \cite{birgin2014practical,brilli2025interior}). On the other hand, the case of projection-based stationarity has been previously considered in \cite{galvan2020alternating,galvan2019convergence}, but a bond between the sequences $\{\tau_k\}$ and $\{\xi_k\}$ is required ($\tau_k\xi_k\to 0$). For this particular setup we thus provide a novel, improved result of convergence that does not need this particular interaction between the two sequences of parameters.
\end{remark}

\section{A derivative-free alternate minimization algorithm for penalty subproblems}\label{sec3}
For addressing problems of the form \eqref{prob_main}, we assumed not to have access to derivatives information. We are thus interested in devising a computationally convenient algorithmic scheme to tackle subproblems \eqref{prob:subprob} in a derivative-free fashion, so that the overall problem can be solved via Algorithm \ref{alg:spm}. The derivative-free setup poses two main challenges: from the one hand, subproblems have to be tackled without relying on gradient based methods; on the other hand, a maybe more subtle issue lies in the inner stopping condition: we cannot directly measure the size of the gradients $\|\nabla_{y,z} P_{\tau_k}(x^k,y^k,z^k)\|$. 

We observe that, for a given value of the penalty parameter $\tau$, it then follows that
\begin{itemize}
 \item [-] fixed $x$, we have two unconstrained separable problems w.r.t.\ $y$ and $z$;
 \item [-] fixed $y$ and $z$, we have a problem in $x$ whose unique globally optimal solution can be determined,  being it strictly convex quadratic with a separable objective function subject to convex constraints; the solution is in fact readily available by projecting the centroid of $y^{\ell+1}$ and $z^{\ell+1}$ (and the other blocks of variables whenever $m>2$) onto the convex set $X$.
\end{itemize}
Similarly as in the previous section, we denote by $k$ the counter of the outer iterations of the sequential penalty approach.
For a fixed $k$, we can then set up an inner optimization loop where we perform derivative-free inexact optimization w.r.t.\ $y$ and $z$ and the exact minimization with respect to $x$. We will denote by $\ell$ the iteration counter of the inner cycles. The algorithm is described in Algorithm \ref{alg:DFAM}. 

\begin{algorithm}[h]
	\caption{\texttt{Derivative-Free Alternate Minimization}}
	\label{alg:DFAM}
	\textbf{Input:}  $\hat{x}^{k-1}, \hat{y}^{k-1}, \hat{z}^{k-1} \in \mathbb{R}^n,$ $\tau_k>0$\\
    Set $\ell=0$\\
    $(x^\ell,y^\ell,z^\ell) = (\hat{x}^{k-1}, \hat{y}^{k-1}, \hat{z}^{k-1})$
    \\ Set $\tilde{\alpha}^\ell_{y}=\tilde{\alpha}^\ell_{z} = (1,\ldots,1)\in\mathbb{R}^n$
    \\\While{\textit{stopping condition not satisfied}}
    {
    $(y^{\ell+1},\tilde{\alpha}^{\ell+1}_{y}) = $ DF-Search($P_{\tau_k}(y,x)$, $x^\ell$, $y^\ell$, $\tilde{\alpha}^{\ell}_{y}$)\label{step:dfs_in_dfam}\\
    $(z^{\ell+1}, \tilde{\alpha}^{\ell+1}_{z}) = $ DF-Search($P_{\tau_k}(z,x)$, $x^\ell$, $z^\ell$, $\tilde{\alpha}^{\ell}_{z}$)\label{step:dfs_in_dfam2}\\
    Set $x^{\ell+1} \in \arg\min_{x\in X}P_{\tau_k}(x,y^{\ell+1},z^{\ell+1})$
    \\Set $\ell=\ell+1$
    }
    \Return $(x^{\ell},y^\ell,z^\ell)$
\end{algorithm}

Clearly, the method revolves around the DF-Search subprocedure, described in Algorithm 
\ref{alg:DFS}. In the latter, coordinate directions are explored according to classical derivative-free line-search principles \cite{Lucidi2002Derivative}: each direction (and its opposite) are polled for a given tentative stepsize, checking a sufficient decrease condition: if the step can be accepted, a further extrapolation is carried out as long as sufficient decrease is attained and the tentative stepsize at the subsequent iteration (for that direction) will be equal to the current actual stepsize; on the other hand, if sufficient decrease is not attained, the direction is discarded altogether for the iteration and the tentative stepsize is decreased for the next one. 

\begin{algorithm}[h]
	\caption{\texttt{Derivative-Free Search (DF-Search)}}
	\label{alg:DFS}
	\textbf{Input:}  $g:\mathbb{R}^n\times \mathbb{R}^n\to \mathbb{R}$, $x^\ell,w^1\in \mathbb{R}^n,$ $\tilde{\alpha}^\ell\in\mathbb{R}^n$,  $\theta\in(0,1)$, $\gamma\in(0,1)$
    \\
    Let $d_1,\ldots,d_n = e_1,\ldots,e_n$\\
    \For{i=1,\ldots,n}{
    \If{$g(x^\ell,w^i+\tilde{\alpha}^\ell_id_i)\le g(x^\ell,w^i)-\gamma(\tilde{\alpha}^\ell_i)^2$}
    {Set $\alpha_i = \max_{j=0,1,\ldots}\{\tilde{\alpha}^\ell_i\theta^{-j}\mid g(x^\ell,w^i+\tilde{\alpha}^\ell_i\theta^{-j}d_i)\le g(x^\ell,w^i)-\gamma(\tilde{\alpha}^\ell_i\theta^{-j})^2\}$
    \\
    Set $w^{i+1} = w^i+\alpha_id_i$
    \\ Set $\tilde{\alpha}_i^{\ell+1} = \alpha_i$}
    \ElseIf{$g(x^\ell,w^i-\tilde{\alpha}^\ell_id_i)\le g(x^\ell,w^i)-\gamma(\tilde{\alpha}^\ell_i)^2$   
    }{Set $\alpha_i = \max_{j=0,1,\ldots}\{\tilde{\alpha}^\ell_i\theta^{-j}\mid g(x^\ell,w^i-\tilde{\alpha}^\ell_i\theta^{-j}d_i)\le g(x^\ell,w^i)-\gamma(\tilde{\alpha}^\ell_i\theta^{-j})^2\}$
    \\
    Set $w^{i+1} = w^i-\alpha_id_i$\\
    Set $\tilde{\alpha}_i^{\ell+1} = \alpha_i$}
    \Else{
    Set $\alpha_i = 0$, $\tilde{\alpha}_i^{\ell+1} = \theta \tilde{\alpha}_i^\ell$
    }
    }
    \Return $w^{n+1},\tilde{\alpha}^{\ell+1}$  
\end{algorithm}

We can now turn to the analysis of the inner derivative-free optimization method. 

\subsection{Convergence analysis}
\label{sec31}
We begin the analysis for Algorithms \ref{alg:DFAM}-\ref{alg:DFS} stating the following preliminary results.
Note that the property, stated for the $y$ block of variables, straightforwardly holds also for $z$. 
\begin{proposition}\label{preliminary_df}
Assume that the gradients of functions $f_j$, $j=1,\ldots,m$, are Lipschitz continuous over $\mathbb{R}^n$. Then, for any $\tau_k>0$, there exist a constants $c_1,c_2>0$ such that for all $\ell$ we have
$$
\Vert \nabla_{y}P_{\tau_k}(x^\ell,y^\ell,z^\ell)\Vert\le (c_1+c_2\tau_k)\max_{i=1,\ldots ,n}\tilde\alpha_{yi}^{\ell+1},
$$
where $\hat{\alpha}_{yi}^{\ell+1}$ denotes $(\tilde{\alpha}_{y}^{\ell+1})_i$.
\end{proposition}
\begin{proof}
Let us consider a generic iteration $\ell$ and let us denote by $w^{\ell,i}$ the intermediate points visited in the DF-Search procedure called at step \ref{step:dfs_in_dfam} in this iteration. Given any $i\in\{1,\ldots ,n\}$, we can distinguish two cases:
\begin{itemize}
 \item [(a)] $w^{\ell,i+1}=w^{\ell, i}$;
 \item[(b)]  $w^{\ell,i+1}\ne w^{\ell, i}$.
\end{itemize}
{\it Case (a)}. From the instructions of Algorithm \ref{alg:DFS}, letting $g(x^\ell,w) = P_{\tau_k}(x^\ell,w,z^\ell)$, we have 
$$
g\left(x^\ell,w^{\ell,i}\pm{{\tilde\alpha_{yi}^\ell}} d_i\right) >
g(x^\ell,w^{\ell,i})-\gamma\left({{\tilde\alpha_{yi}^\ell}}\right)^2,
$$
i.e., since $\tilde{\alpha}_{yi}^{\ell+1} = \theta \tilde{\alpha}_{yi}^\ell$,
$$
g\left(x^\ell,w^{\ell,i}\pm{{\tilde\alpha_{yi}^{\ell+1}}\over{\theta}} d_i\right) >
g(x^\ell,w^{\ell,i})-\gamma\left({{\tilde\alpha_{yi}^{\ell+1}}\over{\theta}}\right)^2.
$$
By applying the Mean Value Theorem in the above inequality in the case of direction $+d_i$, we have
\begin{gather}
    \label{mv1}
\nabla_w g(x^\ell,u_i^\ell)^Te_i\ge -\gamma\left({{\tilde\alpha_{yi}^{\ell+1}}\over{\theta}}\right), \quad u_i^\ell=w^{\ell,i}+t_i^\ell{{\tilde\alpha_{yi}^{\ell+1}}\over{\theta}} d_i,\quad t_i^\ell\in (0,1)
\end{gather}
Recalling that functions $f_j$ are $L$-smooth and that $g(x^\ell,w) = f_1(w)+\frac{\tau_k}{2}\|w-x^\ell\|^2 + C$ where $C$ is some constant, we can deduce that $g$ also has also gradients $\nabla_w g(x^\ell,w)$ that are Lipschitz continuous for some constant $L_2\le L+\tau_k$ (independent of $x^\ell$).
We can then write
\begin{align*}
   \nabla_w g(x^\ell,u_i^\ell)^Te_i
&=\left(\nabla_w g(x^\ell,u_i^\ell)^Te_i-\nabla_w g(x^\ell,y^{\ell})^Te_i+\nabla_w g(x^\ell,y^{\ell})^Te_i\right)
\\& \le
\nabla_{w_i} g(x^\ell,y^{\ell})+\Vert \nabla_w g(x^\ell,u_i^\ell)-\nabla_w g(x^\ell,y^{\ell})\Vert \\& \le
 \nabla_{w_i} g(x^\ell,y^{\ell})+(L+\tau_k)\Vert u_i^\ell-y^\ell \Vert  \\&\le \nabla_{w_i} g(x^\ell,y^{\ell})+(L+\tau_k)\Vert w^{\ell,i}-y^\ell\Vert +(L+\tau_k){{\tilde\alpha_{yi}^{\ell+1}}\over{\theta}}\\& \le \nabla_{w_i} g(x^\ell,y^{\ell})+n(L+\tau_k)\max_{h=1,\ldots ,n}\{\tilde\alpha_{yh}^{\ell+1}\} +(L+\tau_k){{\tilde\alpha_{yi}^{\ell+1}}\over{\theta}} \\& \le \nabla_{w_i} g(x^\ell,y^{\ell})+\frac{(n+1)(L+\tau_k)}{\theta}\max_{h=1,\ldots ,n}\{\tilde\alpha_{yh}^{\ell+1}\},  
\end{align*}
so that, by \eqref{mv1}, it follows
$$
 \nabla_{w_i} g(x^\ell,y^{\ell})\ge -\frac{(n+1)(L+\tau_k)+\gamma}{\theta}\max_{h=1,\ldots ,n}\{\tilde\alpha_{yh}^{\ell+1}\}.
$$
By an analogous reasoning for the case of direction $-d_i$ we obtain
$$
 \nabla_{w_i} g(x^\ell,y^{\ell})\le \frac{(n+1)(L+\tau_k)+\gamma}{\theta}\max_{h=1,\ldots ,n}\{\tilde\alpha_{yh}^{\ell+1}\}.
$$
We can therefore conclude that
\begin{equation}\label{mmm1}
|\nabla_{w_i} g(x^\ell,y^{\ell})|  = \vert\nabla_{y_i} P_{\tau_k}(x^\ell,y^{\ell},z^\ell)\vert\le \frac{(n+1)(L+\tau_k)+\gamma}{\theta}\max_{h=1,\ldots ,n}\{\tilde\alpha_{yh}^{\ell+1}\}.
\end{equation}

\smallskip
\noindent{\it Case (b)}. Again, from the instructions of Algorithm \ref{alg:DFS} and assuming without loss of generality that sufficient decrease is obtained along $+d_i$, we have
\begin{align*}
    g(x^\ell,w^{\ell,i}+\tilde\alpha_{yi}^{\ell+1} d_i)&\le
g(x^\ell,w^{\ell,i})-\gamma\left(\tilde\alpha_{yi}^{\ell+1}\right)^2\\&\le
g(x^\ell,w^{\ell,i})+\gamma\left({{\tilde\alpha_{yi}^{\ell+1}}\over{\theta}}\right)^2,
\end{align*}
and also
$$
g\left(x^\ell,w^{\ell,i}+{{\tilde\alpha_{yi}^{\ell+1}}\over{\theta}} d_i\right) >
g(x^\ell,w^{\ell,i})-\gamma\left({{\tilde\alpha_{yi}^{\ell+1}}\over{\theta}}\right)^2.
$$
By the Mean Value Theorem, we get
\begin{gather*}
    \nabla_w g(x^\ell,u_i^\ell)^Te_i\le \gamma\left({{\tilde\alpha_{yi}^{\ell+1}}\over{\theta}}\right),\quad u_i^\ell=w^{\ell,i}+t_i^\ell{{\tilde\alpha_{yi}^{\ell+1}}} e_i, \quad t_i^\ell\in(0,1),\\\nabla_w g(x^\ell,v_i^\ell)^Te_i>-\gamma\left({{\tilde\alpha_{yi}^{\ell+1}}\over{\theta}}\right),\quad v_i^\ell=w^{\ell,i}+q_i^\ell{{\tilde\alpha_{yi}^{\ell+1}}\over{\theta}} e_i, \quad q_i^\ell\in(0,1).
\end{gather*}
Following a similar reasoning as in case (a), we obtain
\begin{gather*}
    \nabla_{w_i} g(x^\ell,y^{\ell})\le \frac{(n+1)(L+\tau_k)+\gamma}{\theta}\max_{h=1,\ldots ,n}\{\tilde\alpha_{yh}^{\ell+1}\},\\\nabla_{w_i} g(x^\ell,y^{\ell})\ge -\frac{(n+1)(L+\tau_k)+\gamma}{\theta}\max_{h=1,\ldots ,n}\{\tilde\alpha_{yh}^{\ell+1}\},
\end{gather*}
i.e.,
\begin{equation}\label{mmm2}
|\nabla_{w_i} g(x^\ell,y^{\ell})|  = \vert\nabla_{y_i} P_{\tau_k}(x^\ell,y^{\ell},z^\ell)\vert\le \frac{(n+1)(L+\tau_k)+\gamma}{\theta}\max_{h=1,\ldots ,n}\{\tilde\alpha_{yh}^{\ell+1}\}.
\end{equation}
Finally, from (\ref{mmm1}) and (\ref{mmm2}),
we may conclude that
$$
\|\nabla_y P_{\tau_k}(y^\ell,z^\ell,x^\ell)\|\le \sqrt{n} \frac{(n+1)(L+\tau_k)+\gamma}{\theta}\max_{h=1,\ldots ,n}\{\tilde\alpha_{yi}^\ell\},
$$
and this concludes the proof.

\end{proof}

We next prove a result concerning the sequences of stepsizes generated by using Algorithms \ref{alg:DFAM}-\ref{alg:DFS}. In particular, we show that if Algorithm \ref{alg:DFAM} did not stop, the stepsizes would go to zero. 

\begin{proposition}\label{dist_0}
Assume that $\ell\to \infty$ in Algorithm \ref{alg:DFAM}.
Then, for all $i=1,\ldots ,n$, the following properties hold:
\begin{gather}
    \label{step1}
\lim_{\ell\to\infty} \alpha_{yi}^\ell=0,\qquad \lim_{\ell\to\infty} \alpha_{zi}^\ell=0,\\
\label{step2}
\lim_{\ell\to\infty} \tilde \alpha_{yi}^\ell=0, \qquad \lim_{\ell\to\infty} \tilde \alpha_{zi}^\ell=0.
\end{gather}
\end{proposition}
\begin{proof}
 The instructions of the algorithm imply  that for all $\ell$ we have
 $$
 P_{\tau_k}(x^{\ell+1},y^{\ell +1},z^{\ell +1})\le
 P_{\tau_k}(x^{\ell},y^{\ell +1},z^{\ell +1})\le
 P_{\tau_k}(x^{\ell},y^{\ell },z^{\ell }),
 $$
 and hence it follows that the sequence
$\{P_{\tau_k}(x^{\ell},y^{\ell },z^{\ell })\}$ converges. Furthermore, by the instructions of Algorithm \ref{alg:DFS}, the update of each individual variable from vectors $y$ and $z$ guarantees a sufficient decrease of $-\gamma {(\alpha_{i}^\ell)}^2$ (this is trivially satisfied even when $\alpha_i^\ell=0$); we can therefore write
$$
P_{\tau_k}(x^{\ell},y^{\ell +1},z^{\ell +1})\le
P_{\tau_k}(x^{\ell},y^{\ell },z^{\ell })-\gamma\displaystyle{\sum_{i=1}^{n}}\left((\alpha_{yi}^\ell)^2
+(\alpha_{zi}^\ell)^2\right).
$$
Taking into account the convergence of $\{P_{\tau_k}(x^{\ell},y^{\ell },z^{\ell })\}$, it follows that \eqref{step1} holds.

We now prove \eqref{step2}; for the sake of simplicity we consider the sequence $\{\tilde \alpha^\ell_y\}$: the result for $\{\tilde \alpha^\ell_z\}$ can be obtained identically. For any $i\in \{1,\ldots ,n\}$, we split the set of indices of iterates $\{0,1,\ldots \}$ into
two subsets $T$ and $\bar
T$ (for notational convenience we omit the dependence on $i$), where:
\begin{itemize}
\item[-] $\ell\in T$ if and only if $\alpha_{yi}^\ell>0$; \item[-] $\ell\in \bar
T$ if and only if $\alpha_{yi}^\ell=0$.
\end{itemize}
For each $\ell\in T$ we have $\tilde\alpha_{yi}^{\ell+1}=\alpha_{yi}^\ell$, and hence,
if $T$ is an infinite subset, from \eqref{step1} it follows
\begin{equation*}
\lim_{\ell\in T,\ell\to\infty}\tilde\alpha_{yi}^{\ell+1}=0.
\end{equation*}
For every $\ell\in \bar T$, let $m_\ell$ be the largest index such that $m_\ell<\ell$ with
$m_\ell\in T$ (we assume $m_\ell=0$ if this index does not exist, i.e., $T$ is empty).
We can write
$$
\tilde\alpha_{yi}^{\ell+1}=\theta^{\ell-m_\ell}\tilde\alpha_{yi}^{m_\ell+1 }\le \tilde\alpha_{yi}^{m_\ell +1}.
$$
For $\ell\in \bar T$ and $\ell\to\infty$  we have that either $m_\ell\to\infty$ (if
$T$ is an infinite subset) or $\ell-m_\ell\to \infty$ (if $T$ is finite). In the former case, we have
$$\lim_{\ell\in \bar{T},\ell\to\infty}\tilde\alpha_{yi}^{\ell+1}\le \lim_{\ell\in \bar{T},\ell\to\infty}\tilde\alpha_{yi}^{m_\ell +1}=\lim_{\ell\in T,\ell\to\infty}\tilde\alpha_{yi}^{\ell +1} =0.$$
In the latter case, letting $\bar{m}$ be the last index in $T$ and recalling that $\theta\in (0,1)$, we have 
$$\lim_{\ell\in \bar{T},\ell\to\infty}\tilde\alpha_{yi}^{\ell+1}=\lim_{\ell\in \bar{T},\ell\to\infty}\theta^{\ell-\bar{m}}\tilde\alpha_{yi}^{\bar{m}+1} = 0.$$
Putting everything together, we get \eqref{step2}.
\end{proof}

Finally, we show the convergence property w.r.t.\ the $x$ block of variables.
\begin{proposition}\label{dec_properties}
Assume that $\ell\to \infty$ in Algorithm \ref{alg:DFAM}.
Then:
$$
    \lim_{\ell\to\infty}\Vert x^\ell-\Pi_X[x^\ell-\nabla_{x}P_{\tau_k}(x^\ell,y^\ell,z^\ell)]=0.
$$
\end{proposition}
\begin{proof}
By the instructions of Algorithm \ref{alg:DFAM}, $\{P_{\tau_k}(x^\ell,y^\ell,z^\ell)\}$ is a decreasing sequence, hence $\{(x^\ell,y^\ell,z^\ell)\}\subseteq\mathcal{L}_{0}^{\tau_k}=\{(x,y,z)\mid P_{\tau_k}(x,y,z)\le P_{\tau_k}(\hat{x}^{k-1}, \hat{y}^{k-1}, \hat{z}^{k-1}) \}$. By Corollary \ref{compactness}, the level set $\mathcal{L}_0^{\tau_k}$ is compact, therefore the sequence $\{(x^\ell,y^\ell,z^\ell)\}$ admits limit points.
Let us now assume, by contradiction, that there exists an infinite subset ${T} $ such that
\begin{equation}\label{conv}
\lim_{\ell \in {T},\ell\to\infty}(x^\ell,y^\ell,z^\ell)=
(\bar x,\bar y, \bar z)
\end{equation}
with
\begin{equation}\label{false}
\Vert \bar x-\Pi_X[\bar x-\nabla_{x}P_{\tau_k}(\bar x,\bar y,\bar z)]\Vert \ge \eta>0.
\end{equation}
From Proposition \ref{dist_0} it follows
\begin{equation}\label{ellp1}
 \begin{aligned}
  \lim_{\ell\in {T},\ell\to\infty} y^{\ell+1}&=\bar y,\\
  \lim_{\ell\in {T},\ell\to\infty} z^{\ell+1}&=\bar z.
    \end{aligned}
\end{equation}
The instructions of Algorithm \ref{alg:DFAM} imply
$$
P_{\tau_{k}}(x^{\ell+1},y^{\ell+1},z^{\ell+1})\le
P_{\tau_{k}}(x^\ell,y^\ell,z^\ell)
$$
and
$$
P_{\tau_{k}}(x^{\ell+1},y^{\ell+1},z^{\ell+1}) \le
P_{\tau_{k}}(x^{\ell}+\alpha_\ell d_x^\ell,y^{\ell+1},z^{\ell+1})\le
P_{\tau_{k}}(x^{\ell},y^{\ell+1},z^{\ell+1}),
$$
where
$$
d_x^\ell= x^\ell-\Pi_X[x^\ell-\nabla_{x}P_{\tau_{k}}(x^\ell,y^{\ell+1},z^{\ell+1})]
$$
and $\alpha_\ell$ is determined by the Armijo-type line-search.
From the properties of the Armijo-type line-search (\cite[Prop.\ 20.4]{grippo2023introduction}) we get
$$
\lim_{\ell\in {T},\ell\to\infty}\nabla_xP_{\tau_{k}}(x^{\ell},y^{\ell+1},z^{\ell+1})^Td_x^\ell=0.
$$
By the properties of the projected gradient direction \cite[Eq.\ (20.15)]{grippo2023introduction}, we also have
$$
\nabla_xP_{\tau_{k}}(x^{\ell},y^{\ell+1},z^{\ell+1})^Td_x^\ell \le -\|d_x^\ell\|^2.
$$
We can therefore conclude that
$$
\lim_{\ell\in {T},\ell\to\infty}
\Vert x^\ell-\Pi_X[x^\ell-\nabla_{x}P_{\tau_{k}}(x^\ell,y^{\ell+1},z^{\ell+1})]\Vert=0.
$$
Then, recalling \eqref{conv}, \eqref{ellp1} and the continuity properties of the projection mapping,
we obtain
$$
\Vert \bar x-\Pi_X[\bar x-\nabla_{x}P_{\tau_{k}}(\bar x,\bar y,\bar z)\Vert =0,
$$
and this contradicts \eqref{false}.
\end{proof}

\subsection{Inner stopping condition}
\label{sec32}
While we proved the asymptotic convergence properties of Algorithm \ref{alg:DFAM}, we are in fact interested in suitably stopping it in finite time to use it within the sequential penalty scheme of Algorithm \ref{alg:spm}.

We are thus asked to devise a suitable stopping condition ensuring that the output $(x^k,y^k,z^k)$ of Algorithm \ref{alg:DFAM} satisfies conditions \eqref{c11}-\eqref{c3}.

To this aim, we shall exploit the result from Proposition \ref{preliminary_df}. We can indeed set as a termination criterion the combination of the following two properties:
\begin{equation}\label{stop_criterion}
\begin{aligned}
    \max\left\{\max_{i=1,\ldots ,n}\tilde\alpha_{yi}^{\ell+1}, \max_{i=1,\ldots ,n}\tilde\alpha_{zi}^{\ell+1}\right\}
 \le&\frac{\xi_k}{\max\{\tau_k,1\}},\\
 \Vert x^{\ell}-\Pi_X[x^{\ell}-\nabla_{x}P_{\tau_k}(x^{\ell},y^{\ell},z^{\ell})]\|\le& \;\xi_k.
\end{aligned}
\end{equation}
Note that the stopping condition should be checked after step \ref{step:dfs_in_dfam2} of Algorihtm \ref{alg:DFAM} and, when satisfied, the point $(x^\ell,y^\ell,z^\ell)$ should be given as output.

We first discuss the second condition. The gradient of $P_{\tau_k}$ w.r.t.\ $x$ variables is accessible, as it only concerns the quadratic penalty terms that we know in analytical form. Thus, the condition is easily checkable. 
Proposition \ref{dec_properties} implies that the criterion is satisfied in a finite number of iterations.

As for the former condition, clearly it is readily checkable too, and by \ref{dist_0} it is also satisfied in a finite number of inner iterations. From Proposition \ref{preliminary_df} we also get
  \begin{equation*}
    \begin{aligned}
        \Vert \nabla_{y}P_{\tau_k}(x^k,y^k,z^k) \Vert \le (c_{1y}+c_2\tau_k)\frac{\xi_k}{\max\{\tau_k,1\}},\\
    \Vert \nabla_{z}P_{\tau_k}(x^k,y^k,z^k) \Vert \le (c_{1z}+c_2\tau_k)\frac{\xi_k}{\max\{\tau_k,1\}}.
    \end{aligned}
    \end{equation*}
    If $\tau_k\ge 1$, we then get
\begin{equation*}
    \begin{aligned}
        \Vert \nabla_{y}P_{\tau_k}(x^k,y^k,z^k) \Vert \le c_{1y}\frac{\xi_k}{\tau_k}+c_2\xi_k\le (c_{1y}+c_2)\xi_k,\\
    \Vert \nabla_{z}P_{\tau_k}(x^k,y^k,z^k) \Vert \le c_{1z}\frac{\xi_k}{\tau_k}+c_2\xi_k\le (c_{1z}+c_2)\xi_k,
    \end{aligned}
    \end{equation*}
    whereas if $\tau_k<1$ we have
    \begin{equation*}
    \begin{aligned}
        \Vert \nabla_{y}P_{\tau_k}(x^k,y^k,z^k) \Vert \le c_{1y}{\xi_k}+c_2\tau_k\xi_k\le (c_{1y}+c_2)\xi_k,\\
    \Vert \nabla_{z}P_{\tau_k}(x^k,y^k,z^k) \Vert \le c_{1z}{\xi_k}+c_2\tau_k\xi_k\le (c_{1z}+c_2)\xi_k,
    \end{aligned}
    \end{equation*}
    
Thus, conditions \eqref{c11}-\eqref{c3} will be satisfied at each iteration $k$ and we may conclude that the convergence result of Proposition \ref{prop_main} holds.

\section{The case of partially separable functions}
\label{sec4}
The approach resulting from the combination of Sections \ref{sec2} and \ref{sec3} could in principle be employed to solve any problem of the form \eqref{prob_main} without using the derivatives of $f$.
However, as mentioned in the introduction of this manuscript, one of the computational motivations of the proposed strategy has to be found in the particular case of functions $f_j$ being coordinate partially separable. Formally, the problem has the form
\begin{equation} \label{eq:sep_problem}
    \begin{aligned}
        \min_{x\in X}\sum_{j=1}^m f_j(x_{S_j}),
    \end{aligned}
\end{equation}
where $S_j\subseteq\{1,\ldots,n\}$ for all $j=1,\ldots,m$.

In this case, the reformulated problem for the penalty decomposition scheme is
\begin{equation*}
    \begin{aligned}
        \min_{\substack{
            x\in X,\\y_j\in\mathbb{R}^{|S_j|}
        }}\;&\sum_{j=1}^m f_j(y_j)\\
        \text{s.t. }&x_{S_j}-y_j = 0\quad \text{for all }j=1,\ldots,m.
    \end{aligned}
\end{equation*}
In this scenario, we have some advantages:
\begin{itemize}
    \item Each duplicate vector of variables is only made of few components.
    \item As a consequence, the size of equality constraints to be penalized is heavily reduced.
    \item Most importantly, in Algorithm \ref{alg:DFAM}, each DF-Search only concerns few variables, massively reducing the cost of derivative-free optimization steps; indeed, line-search base derivative-free approaches are know to scale poorly as the number of variables grow. The cost of an iteration of Algorithm \ref{alg:DFAM} is now $\sum_{j=1}^{m} 2|S_j|$ evaluations of individual functions $f_j$, which is equivalent to the cost of carrying out an iteration of the DF-Search procedure directly on the original problem;
    \item In addition, the subproblems with respect to each block of variables $y_j$ are completely independent of each other; thus, optimization can naturally take advantage of parallel computation.
\end{itemize}
Note that the original problem should be harder to solve by (derivative-free) alternate minimization as functions $f_j$ are not separable. The variables of the overall problem are thus connected to each other by more complex functions than the penalty terms appearing in the subproblems.

\section{Computational experiments}
\label{sec5}
We report in this section an experimental evaluation of the proposed penalty decomposition derivative free (\texttt{PDDF}) method. We consider a set of problems of the form \eqref{eq:sep_problem}, where the functions $f_j$ are partially separable, by extending some test problems from \texttt{CUTEst} \cite{Gould2015}.
Similarly to what is done in \cite{PorcelliToint2022Exploiting}, we obtain the test problems by taking the summation of several copies of base functions, where a subset of variables is shared between functions, following different patterns. In \texttt{ARWHEAD} a variable is common to all the functions in the summation. In \texttt{BROWNAL6}, \texttt{BROYDN3D}, \texttt{ENGVAL}, \texttt{FREUROTH}, \texttt{MOREBV} and \texttt{TRIDIA} variables are shared between consecutive functions, whereas in \texttt{DIXMAANA}, \texttt{DIXMAANI}, \texttt{NZF1} and \texttt{WOODS} variables are shared following a more intricate rule. In \texttt{BDQRTIC} we find both variables shared between consecutive functions and a variable that is shared among them all. In \texttt{BEALES}, \texttt{POWSING} and \texttt{ROSENBR} instead, no variable is shared between the functions, thus each term of the summation can be optimized independently. The other problems listed in Table \ref{tab:prob_recap} are constructed by taking $m$ copies of each base \texttt{CUTEst} problem and linking them sequentially, so that the last variable of one copy is shared with the first variable of the next one. Additional problems have been generated considering the summation of different \texttt{CUTEst} functions and sharing the first variables between those. The details of the problems constructed in this latter way are available in Table \ref{tab:prob_recap_2}. The base \texttt{CUTEst} problems that we used in this work comprise academic, modeling and real-world problems, both unconstrained and with box constraints. We consider different dimensionalities $n$ for the problems, ranging from 4 to 1000. In total we have 231 unconstrained problems and 44 constrained ones. 

\begin{table*}[t]
	\centering
	\scriptsize
	\caption{Details of the problems employed in the experimental analysis. The problems are obtained taking the summation of several copies of the base \texttt{CUTEst} problem reported and sharing a subset of variables between functions.The problem name also reports whether the problem is unconstrained [\texttt{UC}] or box-constrained [\texttt{B}].}
	\label{tab:prob_recap}
	\renewcommand{\arraystretch}{1.4}%
    \resizebox{\linewidth}{!}{
	\begin{tabular}{|l||c|c|c@{\hspace{0.1cm}}|l||c|c|}%
		\cline{1-3}\cline{5-7}
		\multirow{2}{*}{\textbf{Problem}} & \multirow{2}{*}{$\mathbf{n}$} & \multirow{2}{*}{$\mathbf{m}$} & & \multirow{2}{*}{\textbf{Problem}} & \multirow{2}{*}{$\mathbf{n}$} & \multirow{2}{*}{$\mathbf{m}$}\\[2mm]
		\cline{1-3}\cline{5-7} \cline{1-3}\cline{5-7}
\texttt{AKIVA [UC]} & 5, 9, 17, 33 & 4, 8, 16, 32  &  & \texttt{HATFLDGLS [UC]} & 97, 193, 385, 769 & 4, 8, 16, 32  \\ \cline{1-3}\cline{5-7}
\texttt{ARWHEAD [UC]} & 10, 50, 100, 500, 1000 & 9, 49, 99, 499, 999  &  & \texttt{HILBERTA [UC]} & 5, 9, 17, 33 & 4, 8, 16, 32  \\ \cline{1-3}\cline{5-7}
\texttt{BA-L1SPLS [UC]} & 225, 449, 897 & 4, 8, 16  &  & \texttt{HILBERTB [UC]} & 37, 73, 145, 289 & 4, 8, 16, 32  \\ \cline{1-3}\cline{5-7}
\texttt{BARD [UC]} & 9, 17, 33, 65 & 4, 8, 16, 32  &  & \texttt{HIMMELBCLS [UC]} & 5, 9, 17, 33 & 4, 8, 16, 32  \\ \cline{1-3}\cline{5-7}
\texttt{BDQRTIC [UC]} & 10, 50, 100, 500, 1000 & 6, 46, 96, 496, 996  &  & \texttt{HIMMELBG [UC]} & 5, 9, 17, 33 & 4, 8, 16, 32  \\ \cline{1-3}\cline{5-7}
\texttt{BEALE [UC]} & 5, 9, 17, 33 & 4, 8, 16, 32  &  & \texttt{HS110 [B]} & 37, 73, 145, 289 & 4, 8, 16, 32  \\ \cline{1-3}\cline{5-7}
\texttt{BEALES [UC]} & 10, 50, 100, 500, 1000 & 5, 25, 50, 250, 500  &  & \texttt{LRIJCNN1 [UC]} & 85, 169, 337, 673 & 4, 8, 16, 32  \\ \cline{1-3}\cline{5-7}
\texttt{BIGGS6 [UC]} & 21, 41, 81, 161 & 4, 8, 16, 32  &  & \texttt{LUKSAN11LS [UC]} & 397, 793 & 4, 8  \\ \cline{1-3}\cline{5-7}
\texttt{BOXBODLS [UC]} & 5, 9, 17, 33 & 4, 8, 16, 32  &  & \texttt{LUKSAN12LS [UC]} & 389, 777 & 4, 8  \\ \cline{1-3}\cline{5-7}
\texttt{BRKMCC [UC]} & 5, 9, 17, 33 & 4, 8, 16, 32  &  & \texttt{LUKSAN13LS [UC]} & 389, 777 & 4, 8  \\ \cline{1-3}\cline{5-7}
\texttt{BROWNAL6 [UC]} & 10, 50, 102, 502 & 2, 12, 25, 125  &  & \texttt{LUKSAN14LS [UC]} & 389, 777 & 4, 8  \\ \cline{1-3}\cline{5-7}
\texttt{BROYDN3D [UC]} & 10, 50, 100, 500, 1000 & 10, 50, 100, 500, 1000  &  & \texttt{LUKSAN21LS [UC]} & 397, 793 & 4, 8  \\ \cline{1-3}\cline{5-7}
\texttt{CHNROSNB [UC]} & 197, 393, 785 & 4, 8, 16  &  & \texttt{MOREBV [UC]} & 12, 52, 102, 502 & 12, 52, 102, 502  \\ \cline{1-3}\cline{5-7}
\texttt{CHNRSNBM [UC]} & 197, 393, 785 & 4, 8, 16  &  & \texttt{NZF1 [UC]} & 13, 39, 130, 650 & 5, 17, 59, 299  \\ \cline{1-3}\cline{5-7}
\texttt{CLUSTERLS [UC]} & 5, 9, 17, 33 & 4, 8, 16, 32  &  & \texttt{POWERSUM [UC]} & 13, 25, 49, 97 & 4, 8, 16, 32  \\ \cline{1-3}\cline{5-7}
\texttt{COOLHANSLS [UC]} & 33, 65, 129, 257 & 4, 8, 16, 32  &  & \texttt{POWERSUMB [B]} & 13, 25, 49, 97 & 4, 8, 16, 32  \\ \cline{1-3}\cline{5-7}
\texttt{DENSCHNA [UC]} & 5, 9, 17, 33 & 4, 8, 16, 32  &  & \texttt{POWSING [UC]} & 20, 52, 100, 500 & 5, 13, 25, 125  \\ \cline{1-3}\cline{5-7}
\texttt{DENSCHNB [UC]} & 5, 9, 17, 33 & 4, 8, 16, 32  &  & \texttt{QING [UC]} & 397, 793 & 4, 8  \\ \cline{1-3}\cline{5-7}
\texttt{DENSCHND [UC]} & 9, 17, 33, 65 & 4, 8, 16, 32  &  & \texttt{QINGB [B]} & 17, 33, 65, 129 & 4, 8, 16, 32  \\ \cline{1-3}\cline{5-7}
\texttt{DEVGLA1 [UC]} & 13, 25, 49, 97 & 4, 8, 16, 32  &  & \texttt{ROSENBR [UC]} & 10, 50, 100, 500, 1000 & 5, 25, 50, 250, 500  \\ \cline{1-3}\cline{5-7}
\texttt{DEVGLA1B [B]} & 13, 25, 49, 97 & 4, 8, 16, 32  &  & \texttt{SNAIL [UC]} & 5, 9, 17, 33 & 4, 8, 16, 32  \\ \cline{1-3}\cline{5-7}
\texttt{DGOSPEC [B]} & 9, 17, 33, 65 & 4, 8, 16, 32  &  & \texttt{SSI [UC]} & 9, 17, 33, 65 & 4, 8, 16, 32  \\ \cline{1-3}\cline{5-7}
\texttt{DIAGIQB [B]} & 37, 73, 145, 289 & 4, 8, 16, 32  &  & \texttt{TESTQUAD [UC]} & 37, 73, 145, 289 & 4, 8, 16, 32  \\ \cline{1-3}\cline{5-7}
\texttt{DIXMAANA [UC]} & 15, 51, 102, 501 & 15, 51, 102, 501  &  & \texttt{TOINTGOR [UC]} & 197, 393, 785 & 4, 8, 16  \\ \cline{1-3}\cline{5-7}
\texttt{DIXMAANI [UC]} & 15, 51, 102, 501 & 15, 51, 102, 501  &  & \texttt{TOINTPSP [UC]} & 197, 393, 785 & 4, 8, 16  \\ \cline{1-3}\cline{5-7}
\texttt{ENGVAL [UC]} & 10, 50, 100, 500, 1000 & 9, 49, 99, 499, 999  &  & \texttt{TOINTQOR [UC]} & 197, 393, 785 & 4, 8, 16  \\ \cline{1-3}\cline{5-7}
\texttt{ERRINROS [UC]} & 197, 393, 785 & 4, 8, 16  &  & \texttt{TRIDIA [UC]} & 10, 50, 100, 500, 1000 & 10, 50, 100, 500, 1000  \\ \cline{1-3}\cline{5-7}
\texttt{ERRINRSM [UC]} & 197, 393, 785 & 4, 8, 16  &  & \texttt{TRIGON1 [UC]} & 37, 73, 145, 289 & 4, 8, 16, 32  \\ \cline{1-3}\cline{5-7}
\texttt{EXPFIT [UC]} & 5, 9, 17, 33 & 4, 8, 16, 32  &  & \texttt{TRIGON1B [B]} & 37, 73, 145, 289 & 4, 8, 16, 32  \\ \cline{1-3}\cline{5-7}
\texttt{FREUROTH [UC]} & 10, 50, 100, 500, 1000 & 9, 49, 99, 499, 999  &  & \texttt{VANDANMSLS [UC]} & 85, 169, 337, 673 & 4, 8, 16, 32  \\ \cline{1-3}\cline{5-7}
\texttt{HATFLDA [B]} & 13, 25, 49, 97 & 4, 8, 16, 32  &  & \texttt{WATSON [UC]} & 45, 89, 177, 353 & 4, 8, 16, 32  \\ \cline{1-3}\cline{5-7}
\texttt{HATFLDB [B]} & 13, 25, 49, 97 & 4, 8, 16, 32  &  & \texttt{WOODS [UC]} & 20, 40, 200, 400 & 30, 60, 300, 600  \\ \cline{1-3}\cline{5-7}
\texttt{HATFLDFLS [UC]} & 9, 17, 33, 65 & 4, 8, 16, 32  &  &  &  &  \\ \cline{1-3}\cline{5-7}
	\end{tabular}
}
\end{table*}

\begin{table*}[t]
	\centering
	\scriptsize
	\caption{Details of the problems employed in the experimental analysis. The problems are obtained taking the summation of different \texttt{CUTEst} problems and sharing the first $s$ variables between functions. The problem name also reports whether the problem is unconstrained [\texttt{UC}] or box-constrained [\texttt{B}].}
	\label{tab:prob_recap_2}
	\renewcommand{\arraystretch}{1.4}%
    \resizebox{0.65\linewidth}{!}{
    \begin{tabular}{|p{10cm}||c|c|c|c|}%
		\hline
		\multirow{2}{*}{\textbf{Problem}} & \multirow{2}{*}{$\mathbf{n}$} & \multirow{2}{*}{$\mathbf{m}$} &  \multirow{2}{*}{$\mathbf{s}$} \\[2mm]
		\hline \hline
\texttt{3PK, CHEBYQAD, DECONVB, SANTALS [B]} & 178 & 4  & 8  \\ \hline
\texttt{AKIVA, BEALE, BOXBODLS, BRKMCC, BROWNBS [UC]} & 6 & 5  & 1  \\ \hline
\texttt{ALLINITU, BROWNDEN, BIGGS6, FBRAIN3LS, STREG [UC]} & 16 & 5  & 2  \\ \hline
\texttt{BA-L1LS, BA-L1SPLS [UC]} & 106 & 2  & 8  \\ \hline
\texttt{BARD, BOX3, CLIFF, CLUSTERLS, CUBE, DJTL, ENGVAL2, EXPFIT, GAUSSIAN, GROWTHLS, HAIRY [UC]} & 17 & 11  & 1  \\ \hline
\texttt{BQPGABIM, BQPGASIM [B]} & 88 & 2  & 8  \\ \hline
\texttt{CHNROSNB, CHNRSNBM [UC]} & 92 & 2  & 8  \\ \hline
\texttt{CHWIRUT1LS, CHWIRUT2LS [UC]} & 4 & 2  & 2  \\ \hline
\texttt{DENSCHNA, DENSCHNB, DENSCHNC, DENSCHND, DENSCHNE, DENSCHNF [UC]} & 9 & 6  & 1  \\ \hline
\texttt{DEVGLA1B, DEVGLA2B [B]} & 7 & 2  & 2  \\ \hline
\texttt{DIAMON2DLS, DIAMON3DLS [UC]} & 157 & 2  & 8  \\ \hline
\texttt{FBRAIN2LS, FBRAINLS [B]} & 4 & 2  & 2  \\ \hline
\texttt{GAUSS1LS, GAUSS2LS, GAUSS3LS [UC]} & 20 & 3  & 2  \\ \hline
\texttt{HATFLDA, HATFLDB, HATFLDC [B]} & 25 & 3  & 4  \\ \hline
\texttt{HEART6LS, HEART8LS [UC]} & 12 & 2  & 2  \\ \hline
\texttt{HS1, HS110, HS2, HS25, HS3, HS38, HS3MOD, HS4, HS5 [B]} & 21 & 9  & 1  \\ \hline
\texttt{HYDC20LS, HYDCAR6LS [UC]} & 112 & 2  & 16  \\ \hline
\texttt{LANCZOS1LS, LANCZOS2LS, LANCZOS3LS [UC]} & 14 & 3  & 2  \\ \hline
\texttt{LUKSAN11LS, LUKSAN12LS, LUKSAN13LS, LUKSAN14LS, LUKSAN15LS, LUKSAN16LS, LUKSAN17LS, LUKSAN21LS, LUKSAN22LS [UC]} & 766 & 9  & 16  \\ \hline
\texttt{MANCINO, QING, SENSORS [UC]} & 236 & 3  & 32  \\ \hline
\texttt{MGH09LS, MGH10LS, MGH10SLS, MGH17LS, MGH17SLS, NELSONLS, OSBORNEA, OSBORNEB [UC]} & 25 & 8  & 2  \\ \hline
\texttt{MISRA1ALS, MISRA1BLS, MISRA1CLS, MISRA1DLS [UC]} & 5 & 4  & 1  \\ \hline
\texttt{PFIT1LS, PFIT2LS, PFIT3LS, PFIT4LS [B]} & 6 & 4  & 2  \\ \hline
\texttt{TOINTGOR, TOINTPSP, TOINTQOR [UC]} & 134 & 3  & 8  \\ \hline
\texttt{TRIGON1, TRIGON2 [UC]} & 18 & 2  & 2  \\ \hline
\texttt{TRIGON1B, TRIGON2B [B]} & 16 & 2  & 4  \\ \hline
\texttt{VESUVIALS, VESUVIOLS, VESUVIOULS [UC]} & 20 & 3  & 2  \\ \hline
	\end{tabular}
    }
\end{table*}

We also consider one problem outside of \texttt{CUTEst} collection, that is the Lockwood problem \cite{gramacy2020surrogates}. In its original formulation,  we aim at minimizing a linear objective with two nonlinear constraints and box constraints, and is formulated as
\begin{equation} \label{eq:lockwood_orig}
\min_{x} \left\{f(x) = \sum_{j=1}^6 x_j : c_1(x) \leq 0, c_2(x) \leq 0, x \in [0, 2 \cdot 10^4]^6\right\}.
\end{equation}
Since the two nonlinear constraints are black boxes, we are not able to use \texttt{PDDF} to solve the Lockwood problem directly. We therefore consider two alternative versions of this problem, that consider the two nonlinear constraints as part of the objective, i.e., in the form of penalties. We will elaborate on this later.

The performance of the proposed \texttt{PDDF} method is evaluated in comparison to the classical line-search based derivative-free coordinate descent method (\texttt{LS}) proposed in \cite{Lucidi2002Derivative}. In \texttt{LS} the objective function sampling is carried out along coordinate directions, making extrapolation steps when a sufficient decrease of the objective function is obtained for the tentative stepsize, ensuring at every step that the solution remains feasible by using a maximal step size. We refer the reader to \cite{Lucidi2002Derivative} for further details. In addition, we compare our algorithm with the NOMAD implementation \cite{audet2022algorithm} of the Mesh Adaptive Direct Search algorithm, a widely used derivative-free optimization method that explores the search space using a set of mesh points whose resolution is adaptively refined to efficiently locate solutions.

The inner termination criterion for \texttt{PDDF} consists in the conditions \eqref{stop_criterion} with $\xi_k = 10^{-2}$ for all $k$. The outer loop is then stopped when either $\|x^{k+1}-x^k\|\le 10^{-2}$ or $k=100$. The \texttt{LS} method is then run starting from the obtained solution for a refinement process. For \texttt{LS}, used either standalone or as a refiner following \texttt{PDDF}  we employ $\max_{i=1, \dots, n} \alpha_i \leq 10^{-4}$ as stopping condition. Similarly, \texttt{MADS} terminates when the mesh size is below $10^{-4}$. In case one algorithm does not meet the termination condition within 10 minutes of wall-clock time we terminate the experiment. The initial stepsize is set to $1$ for both \texttt{PDDF} and \texttt{LS}, while MADS uses the default initial mesh size. As regards the line-search, both \texttt{PDDF} and \texttt{LS} use $\gamma = 10^{-6}$ and $\theta=0.5$. For \texttt{PDDF} we set $\tau_0 = 1$ and $\tau_{k+1} = \min\{1.1 \tau_{k}, 10^{8}\}$. When considering constrained problems, $\texttt{PDDF}$ handles the box explicitly in the line-search procedure, as \texttt{LS} does. The computational experiments here reported were implemented in \texttt{Python3} and executed on a computer running Ubuntu 24.04 with an Intel Core i5-13400F 10 cores 4.60 GHz processor and 32 GB RAM. 
The implementation of the proposed \texttt{PDDF} method is available at \href{https://github.com/dadoPuccio/PD_derivative_free}{\texttt{github.com/dadoPuccio/PD\_derivative\_free}}.

In order to offer a fair comparison between the three approaches, we count the number of function evaluations of each function $f_j$ that composes the objective. By assuming that the computational burden for evaluating each $f_j$ is similar, the total of individual sub-function evaluations provides an unbiased metric to compare \texttt{PDDF}, \texttt{LS} and \texttt{MADS}.  As a tool for comparing the behavior of the algorithm considered we employ both data profiles \cite{more2009dataprofile} and performance profiles \cite{Dolan2002}. Data profiles report the percentage of solved problems as the computational budget increases, while, given a metric, performance profiles report the fraction of fraction of problems solved within a given performance ratio with respect to the best solver. We account a problem as solved by the solver $S$ in case $f(x_0)-f(x_S) \geq (1-\epsilon)(f(x_0)-f(x_{\text{best}}))$, where $f(x_0)$ is the objective at the starting point, $x_S$ is the solution found by the solver $S$ and $x_{\text{best}}$ is the best solution found by all optimizers. In data profiles we normalize the number of function evaluations in problems with different dimension $n$ and different number of functions $m$ by considering groups of $m \times (n+1)$ sub-function evaluations. All problems are initialized to the \texttt{CUTEst} default starting solution for all solvers.

\begin{figure}[ht!]
	\subfloat[Data profile in the unconstrained problems with $\epsilon=0.01$]{\includegraphics[width=0.45\textwidth]{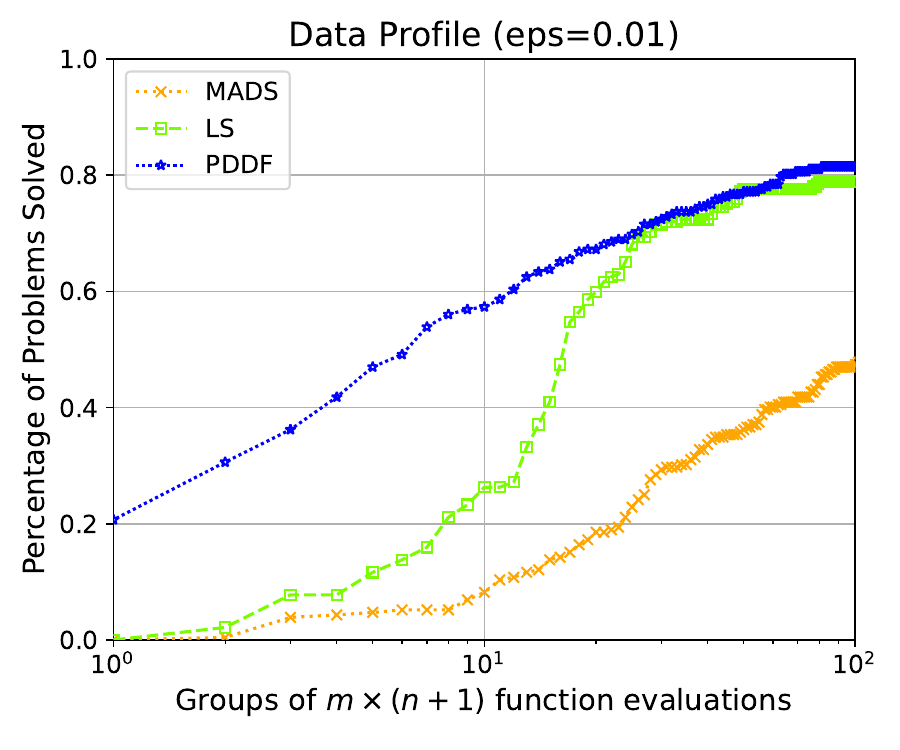}}
    \hfil
	\subfloat[Data profile in the box-constrained problems with $\epsilon=0.01$]{\includegraphics[width=0.45\textwidth]{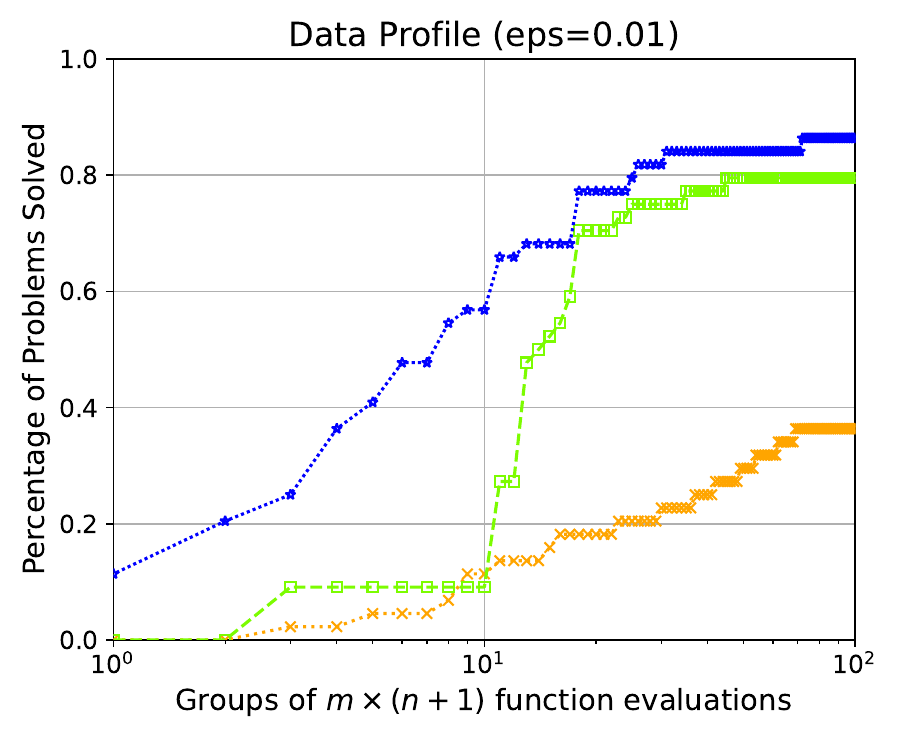}}
    \\[3mm]
    \subfloat[Data profile in the unconstrained problems with $\epsilon=10^{-4}$]{\includegraphics[width=0.45\textwidth]{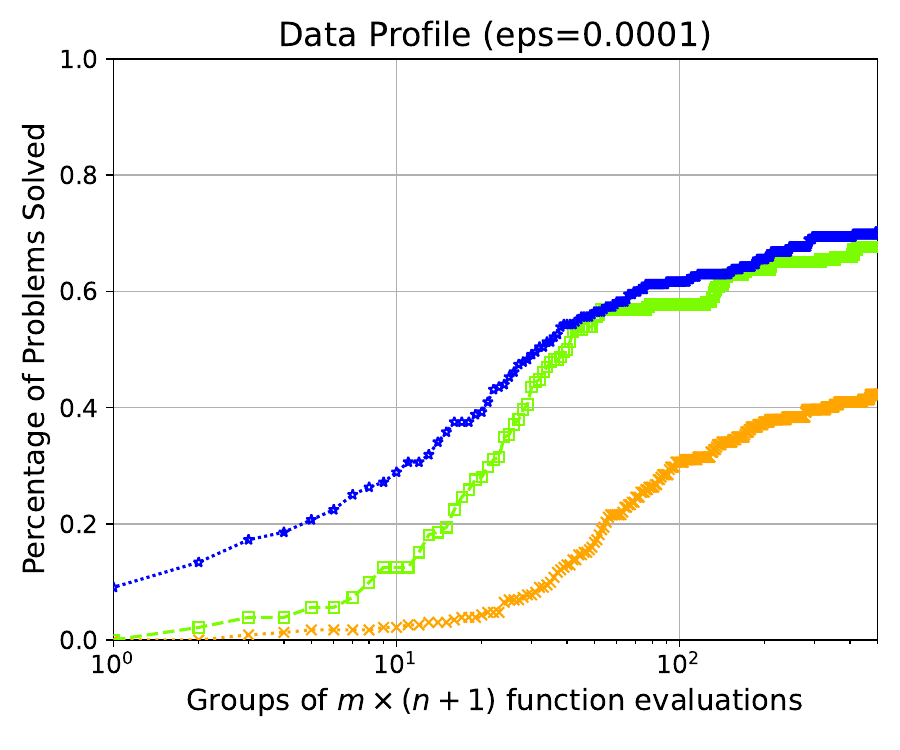}}
    \hfil
	\subfloat[Data profile in the box-constrained problems with $\epsilon=10^{-4}$]{\includegraphics[width=0.45\textwidth]{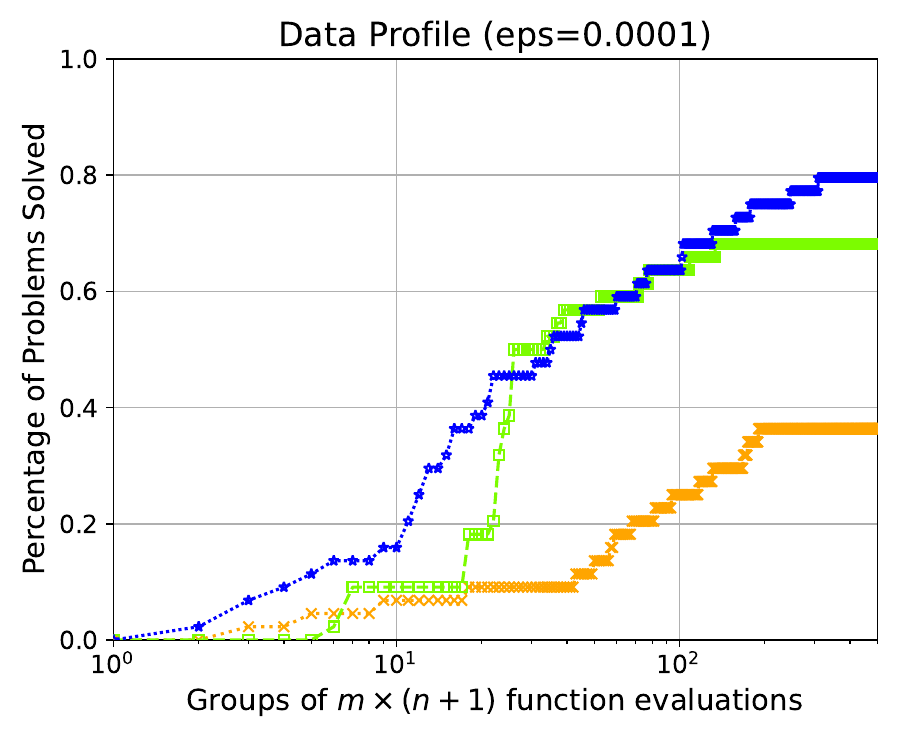}}
    \\[3mm]
    \subfloat[Performance profile of the number of sub-function evaluations in the unconstrained problems]{\includegraphics[width=0.45\textwidth]{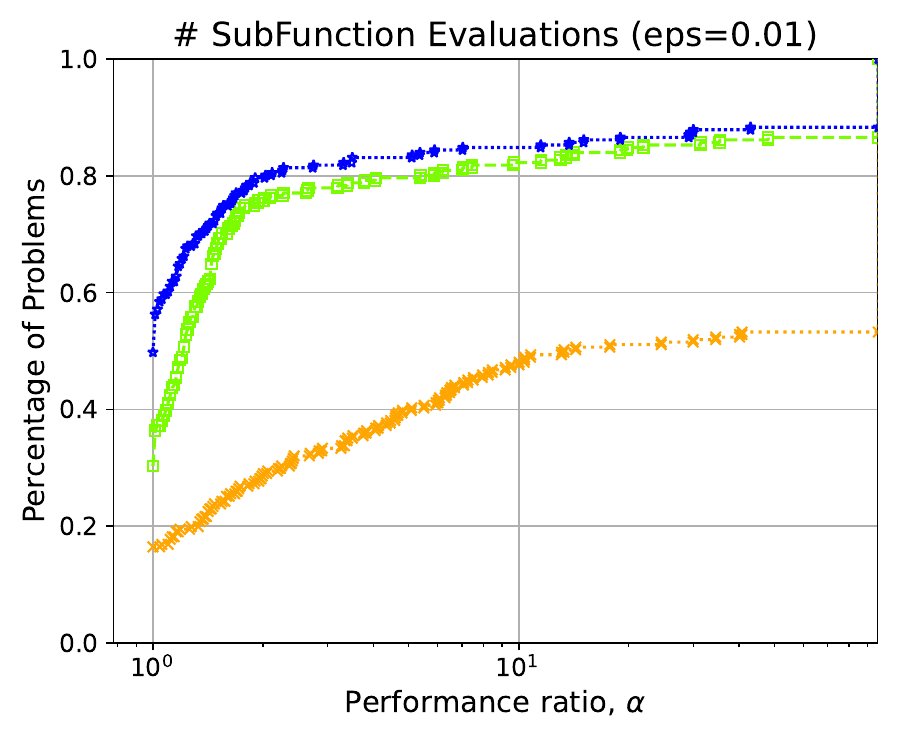}}
    \hfil
	\subfloat[Performance profile of the number of sub-function evaluations in the box-constrained problems]{\includegraphics[width=0.45\textwidth]{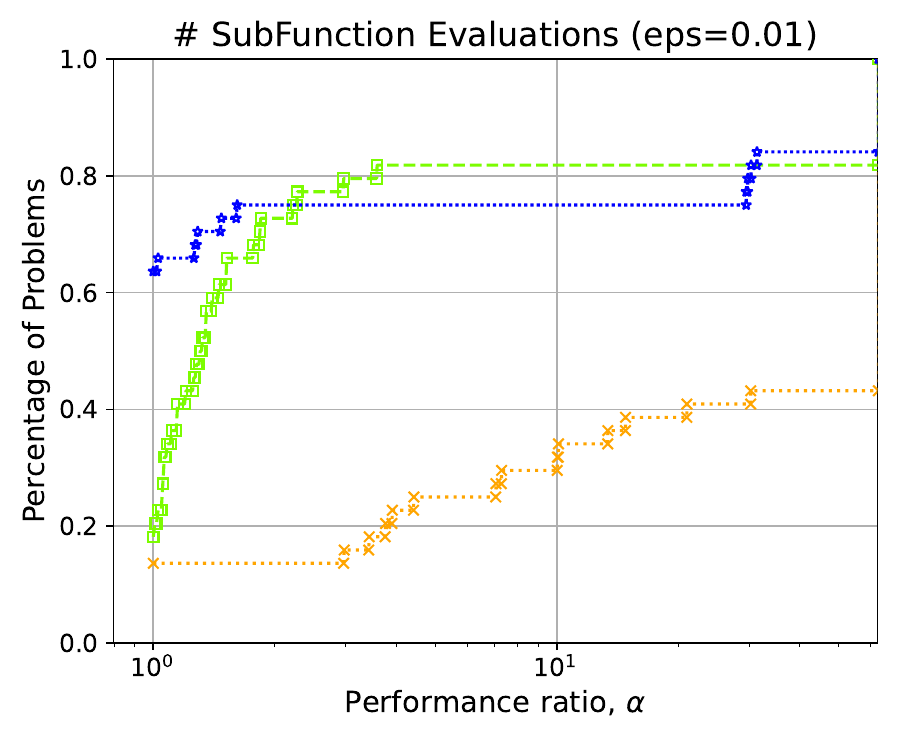}}
	\caption{Data profiles and performance profiles in the problems derived from \texttt{CUTEst}.}
	\label{fig:profiles}
\end{figure}

In Figure \ref{fig:profiles} we report the data profiles and performance profiles in the set of problems derived from \texttt{CUTEst}. We consider unconstrained and bound-constrained problems separately in the profiles. In both cases we can observe that \texttt{PDDF} is often capable of obtaining good quality solutions with fewer function evaluations, especially if the required precision is not extremely high (as often happens in black-box settings). On the contrary \texttt{MADS} is far behind the others, probably due to the large problem dimensionalities considered in our experimental setting. 

\begin{figure}[t]
    \subfloat[Time performance profile in the unconstrained problems]{\includegraphics[width=0.45\textwidth]{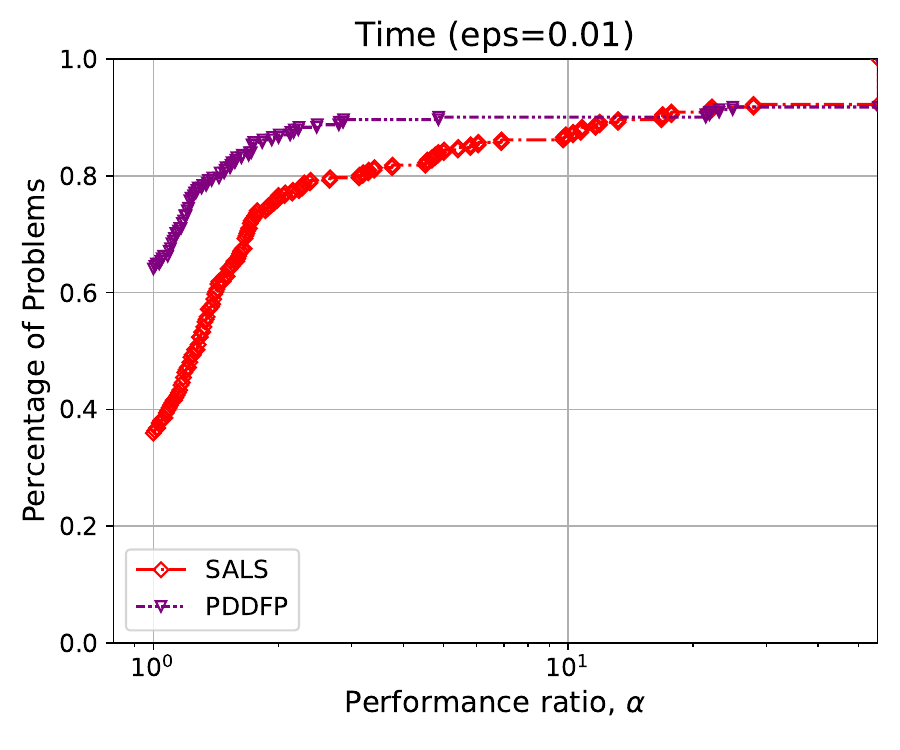}}
    \hfil
	\subfloat[Time performance profile in the box-constrained problems]{\includegraphics[width=0.45\textwidth]{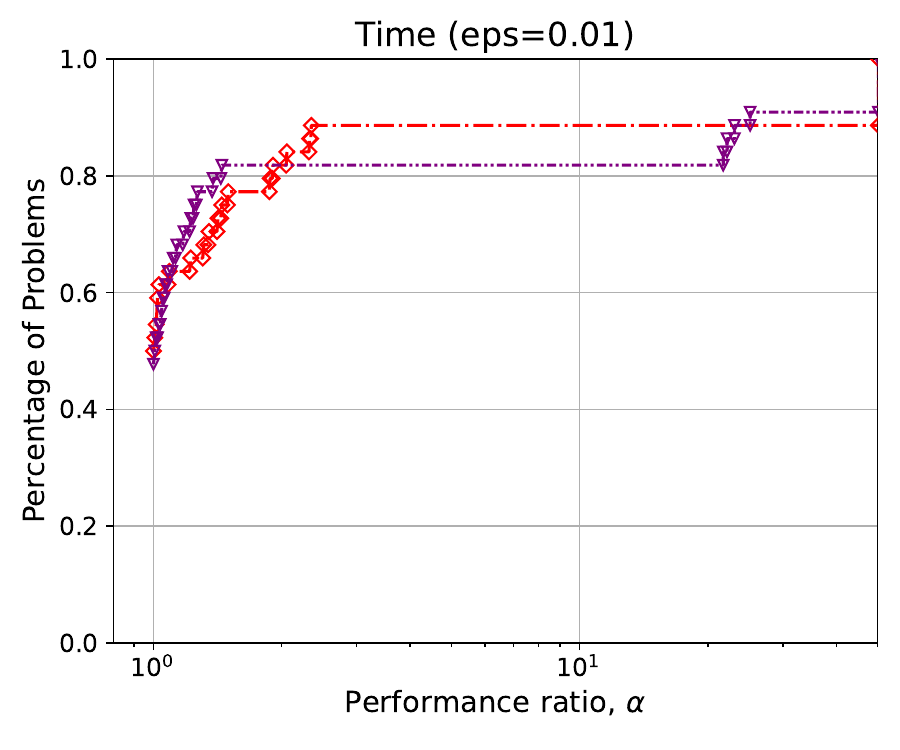}}
	\caption{Performance profile of the wall-clock time require do reach the termination condition in the problems derived from \texttt{CUTEst}.}
	\label{fig:time_profile}
\end{figure}

We provide now the reader with some results concerning the wall-clock time that \texttt{PDDF} method employs to converge. Given the structure of the subproblems obtained with the penalty decomposition scheme, each duplicate vector of variables can be optimized independently of each other: therefore, a parallel implementation of the penalty decomposition derivative free method (\texttt{PDDFP}) can be fruitfully employed. This is particularly appealing for real-world black-box problems, where objective function evaluations are typically expensive to compute; hence, being able of evaluating multiple functions $f_j$ in parallel allows for considerable improvements to the overall computational time. The scaling properties of \texttt{PDDFP} are tightly linked to the number of sub-functions that compose an objective function. Given duplication of variables for each sub-function, adding a processor will be beneficial for \texttt{PDDFP} only in case the total number of computing units stays below the number of sub-functions.

The \texttt{LS} method, on the other hand, cannot exploit parallel computation; for a more challenging test, \texttt{PDDFP} is thus compared with a structure-aware implementation of the classical line-search based derivative-free coordinate descent method (\texttt{SALS}): by keeping track of the values of the individual functions $f_j$, it is possible to evaluate only a subset of the functions when a new point is considered, i.e., those impacted by changes to the currently considered variable. 

As regards the parallel implementation, we limit the maximum number of workers to 12. We simulate the high cost of function evaluations by introducing a 1ms pause at each evaluation. In \ref{fig:time_profile} we report the outcome of this analysis in form of performance profiles. We observe that \texttt{PDDFP} is superior to \texttt{SALS} especially in unconstrained problems. For bound-constrained problems, the gap between the two algorithms is smaller because there are fewer cases with a large number of sub-functions $m$, which limits \texttt{PDDFP}'s ability to fully exploit the available parallel workers.

In the final part for our computational analysis, we compare the behavior of the considered algorithms in two different modification of the Lockwood problem \cite{gramacy2020surrogates}. As earlier discussed in this Section, the original formulation of the problem cannot be handled by \texttt{PDDF} because of the structure of its constraints. We devise and adapted variant of \eqref{eq:lockwood_orig} where the constraints are moved to the objective function in the form of static penalties, leading to 
\begin{equation} \label{eq:lockwood1}
    \min_{x} \left\{f(x) = \sum_{j=1}^6 x_j +  \max\{c_1(x),0\} + \max\{c_2(x), 0\} : x \in [0, 2 \cdot 10^4]^6\right\}, \tag{L1}
\end{equation}
leading to a 6-dimensional problem with box-constraints and an objective function made up of three terms. The resulting problem does fit in our framework but the three sub-function that compose the objective are not partially separable. We therefore consider a partially separable version of the problem \eqref{eq:lockwood1} by considering four copies of its objective function with different (random) weights assigned to the three terms, and sharing the last variable of one copy as the first variable of the next one. The resulting problem, that we indicate with (L2), is a 21-dimensional box-constrained problem with a partially separable objective function mad up of four terms.

\begin{figure}[t]
    \subfloat{\includegraphics[width=0.45\textwidth]{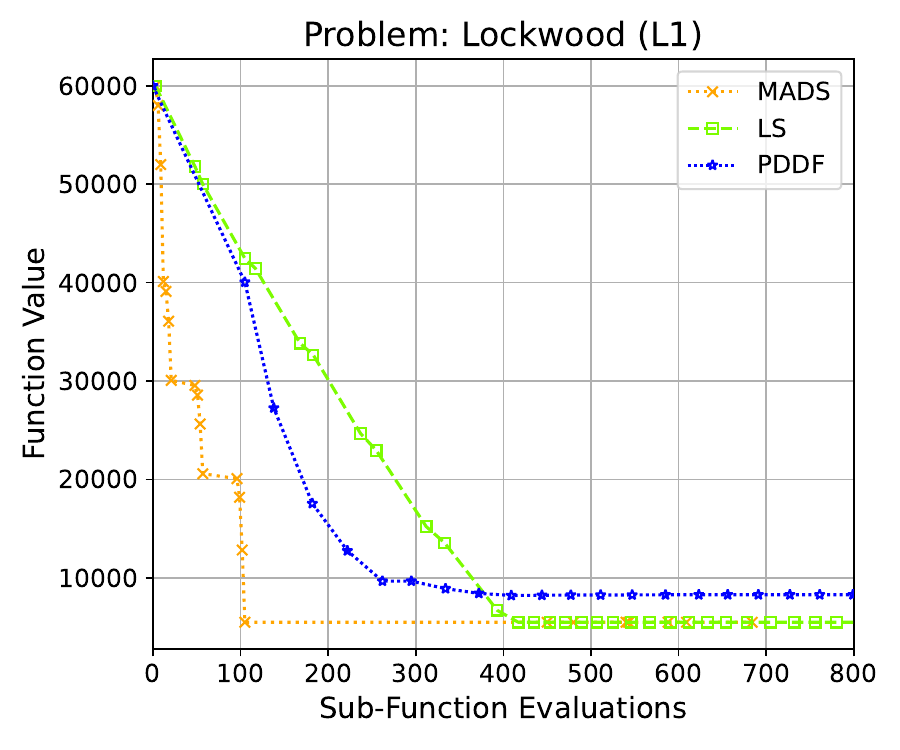}}
    \hfil
	\subfloat{\includegraphics[width=0.45\textwidth]{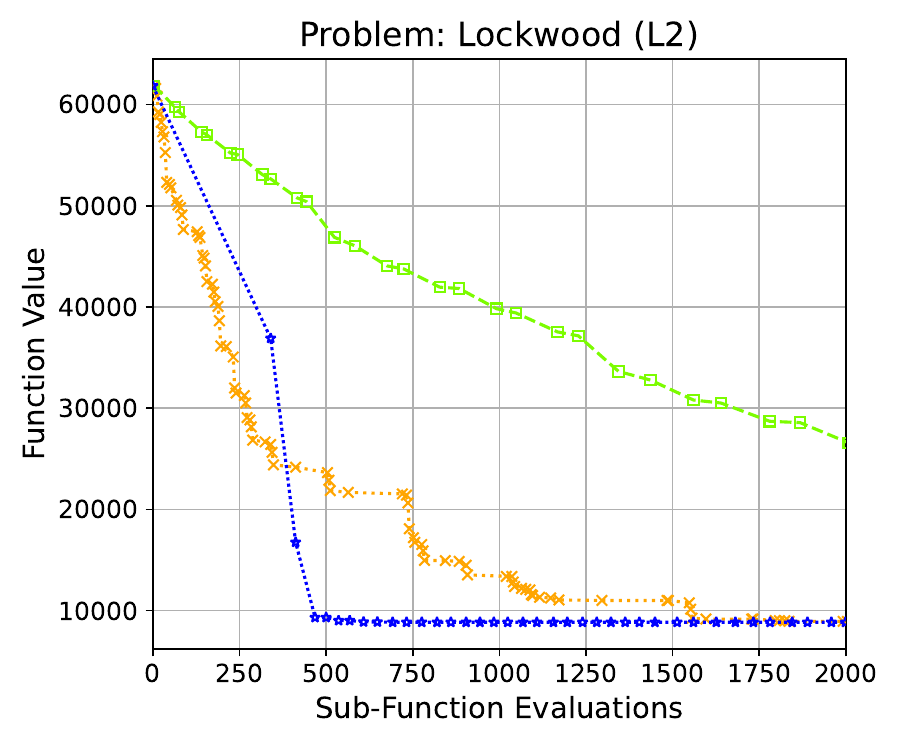}}
	\caption{Evolution of the objective function value with respect to the total number of sub-function evaluations for the compared optimization algorithms in the two versions of the Lockwood problem considered.}
	\label{fig:lockwood}
\end{figure}

In this case the termination condition of \texttt{PDDF} is $\xi_k = 10^{-1}$ and $\|x^{k+1} -x^k\| \leq 10^{-1}$, followed by \texttt{LS} to further refine the solution. 
For \texttt{LS} we employ $\max_{i=1, \dots, n} \alpha_i \leq 10^{-2}$ as the termination criteria, while \texttt{MADS} terminates when the mesh size is below $10^{-2}$.
In Figure \ref{fig:lockwood} we report the evolution of the objective function during the optimization when considering \texttt{PDDF}, \texttt{LS} and \texttt{MADS}. In the case of (L1) the penalty decomposition scheme of \texttt{PDDF} is disfavored by the not-partially separable nature of the problem, therefore leading to the worst performing algorithm in this case. On the other hand in problem (L2), \texttt{PDDF} is by far the most effective as it is capable of converging with fewer sub-function evaluations, proving once more that when the function has a partially separable structure the proposed algorithm is particularly effective.

\section{Conclusions}
\label{sec6}
In this paper, we described and formally analyzed a penalty decomposition scheme for derivative-free optimization of coordinate partially separable finite-sum problems. The approach is proved to be well defined and possess global convergence guarantees under very mild standard assumptions. The results of the computational experiments indicate that the proposed penalty decomposition algorithm is capable of outperforming classical coordinate search methods and more elaborate mesh adaptive approaches in large scale structured problems. Moreover, the approach appears to be well-suited for parallelization; in scenarios where multi-processing is possible and the cost of computation is dominated by the evaluation of the black-box function, the proposed framework is shown to be superior also to a structure-informed smart implementation of the line-search based coordinate descent method. A deeper study in this setting might thus be of interest in future research.

The nice feedback obtained by our experiments opens up to interesting paths that could be explored in future research. First, working independently with each term in the objective should make the entire process less vulnerable to the presence of noise in a subset of the objectives. 
Moreover, the separability with respect to the sum terms is particularly nice in mixed setting where derivatives are available for some of the objectives \cite{liuzzi2012decomposition}; indeed, gradient based optimization steps could be carried out w.r.t.\ blocks of variables associated with white-box terms; this would not be possible directly tackling the original formulation of the problem, where the same variables possibly appear in white and black box terms simultaneously.   
A further strength point of penalty decomposition approaches in general is that they are able to handle additional constraints with a limited overhead (see, e.g., the discussion in \cite{kanzow2023inexact}): indeed, additional constraints can be moved to the penalty function and handled naturally within the sequential penalty scheme.
Finally, the development of analogous penalty schemes based on the augmented Lagrangian function \cite{birgin2014practical}, instead of the simple quadratic penalty, could be considered; in particular, under convexity assumptions for the objective function, ADMM-type \cite{neal2011distributed} methods might be appealing to solve problem \eqref{prob:vincolato}.

\section*{Declarations}

\subsection*{Funding}
No funding was received for conducting this study.

\subsection*{Disclosure statement}

The authors report there are no competing interests to declare.

\subsection*{Code availability statement}
The implementation of the proposed \texttt{PDDF} method and its parallel variant is available at \href{https://github.com/dadoPuccio/PD_derivative_free}{\texttt{github.com/dadoPuccio/PD\_derivative\_free}}.

\subsection*{Notes on contributors}
\textbf{Francesco Cecere} was born in Frosinone in 2000. He received his Bachelor’s degree in Mathematics in 2022 at Sapienza University of Rome. He then received his Master’s degree in Applied Mathematics in 2025 at the same university. He was an exchange student in 2023 at ENS Lyon. 
\\\\
\textbf{Matteo Lapucci} received his bachelor and master's degree in Computer Engineering at the University of Florence in 2015 and 2018 respectively. He then received in 2022 his PhD degree in Smart Computing jointly from the Universities of Florence, Pisa and Siena. He is currently Assistant Professor at the Department of Information Engineering of the University of Florence. His main research interests include theory and algorithms for sparse, multi-objective and large scale nonlinear optimization.
\\\\
\textbf{Davide Pucci} is a PhD student at the Global Optimization Laboratory of the University of Florence. He obtained both his Bachelor's and Master's degrees from the same university, in 2020 and 2023, respectively. His research focuses on optimization methods for machine learning and multi-objective optimization.
\\\\
\textbf{Marco Sciandrone}  is a full professor of Operations Research at Sapienza University of Rome. His research interests include operations research, nonlinear optimization, and machine learning. He has published more than 60 papers on international journals, and is coauthor of the book Introduction to Methods for Nonlinear Optimization, Springer, 2023.

\subsection*{ORCID}

Matteo Lapucci: 0000-0002-2488-5486\\
Davide Pucci: 0000-0003-4424-9185\\
Marco Sciandrone: 0000-0002-5234-3174

\bibliographystyle{tfs}
\bibliography{bibliography}

\end{document}